\RequirePackage{fix-cm}
\documentclass[smallextended]{svjour3}       
\smartqed  
\usepackage{graphicx}
\usepackage{amsmath,amsfonts,amssymb,tikz-cd,cite}
\begin{document}

\title{Closed Affine Manifolds with an Invariant Line
\thanks{The author gratefully acknowledges research support from NSF Grant Award Number 1709791}
}


\author{Charles Daly        
}


\institute{Charles Daly \at
              4176 Campus Dr, College Park, MD 20742 \\
              \email{cdaly69@umd.edu}           
}

\date{Received: date / Accepted: date}

\maketitle


\begin{abstract}
A closed affine manifold is a closed manifold with coordinate patches into affine space whose transition maps are restrictions of affine automorphisms.  Such a structure gives rise to a local diffeomorphism from the universal cover of the manifold to affine space that is equivariant with respect to a homomorphism from the fundamental group to the group of affine automorphisms.  The local diffeomorphism and homomorphism are referred to as the developing map and holonomy respectively.  In the case where the linear holonomy preserves a common vector, certain `large' open subsets upon which the developing map is a diffeomorphism onto its image are constructed.  A modified proof of the fact that a radiant manifold cannot have its fixed point in the developing image is presented.  Combining these results, this paper addresses the non-existence of certain closed affine manifolds whose holonomy leaves invariant an affine line.  Specifically, if the affine holonomy acts purely by translations on the invariant line, then the developing image cannot meet this line.
\end{abstract}

\newpage


\section{Affine Structures Preliminaries}
\label{sec:1}
This section is largely dedicated to preliminary notions regarding closed affine manifolds.  Specifically, this section provides examples, notation, and some basic results regarding the developing map and holonomy.
%
%
%
\begin{definition}
An $n$-dimensional affine manifold is a $n$-dimensional manifold $M$ equipped with charts $(U_{\alpha}, \Phi_{\alpha})$ where each $\Phi_{\alpha} : U_{\alpha} \longrightarrow \mathbb{A}^{n}$ is a diffeomorphism onto an open subset of affine $n$-space such that the restriction of each transition map on each connected component of $U_{\alpha}\cap U_{\beta}$ is an affine automorphism.  Explicitly, this says for each pair of charts $(U_{\alpha}, \Phi_{\alpha})$  and $(U_{\beta}, \Phi_{\beta})$ and each connected component $V \subset U_{\alpha}\cap U_{\beta}$, there exists an affine automorphism $A_{\beta\alpha,V} : \mathbb{A}^{n} \longrightarrow \mathbb{A}^{n}$ so that the following equality holds
%
%
%
%
\begin{equation}\label{eq1}
	\begin{tikzcd}
	& V \arrow{dl}[swap]{\Phi_{\alpha}} \arrow{dr}{\Phi_{\beta}} & \\
	\Phi_{\alpha}(V) \arrow{rr}[swap]{\Phi_{\beta}\circ\Phi_{\alpha}^{-1} = A_{\beta\alpha}|_{V}} & 		& \Phi_{\beta}(V)
	\end{tikzcd} \nonumber
\end{equation}
\end{definition}
This definition lends itself to a natural generalization of what is known as a $(G,X)$-manifold, where $G$ is a lie group acting strongly effectively on a manifold $X$.  The study of affine geometry is the study of $(\text{Aff}(n,\mathbb{R}), \mathbb{A}^{n})$-manifolds.  Here are some standard examples in the literature of affine structures on the two-torus.  
%
%
%
%
\begin{example}\label{ex1}
Let $T$ be a rank two free abelian subgroup of $\text{Aff}(2,\mathbb{R})$ which acts on the affine plane by translations.  Let $M = \mathbb{A}^{2}/T$.  Since $T$ acts properly and freely on $\mathbb{A}^{2}$, the associated quotient map is a smooth covering and the charts on $M$ are naturally diffeomorphisms onto open subsets of $\mathbb{A}^{2}$.  The transition maps are given by elements of $T$, and thus $M$ inherits an affine structure.  In fact, quite a bit more can be said.  Since the group of translations preserve the standard euclidean metric, $M$ inherits a riemannian structure.  This provides $M$ a riemannian metric locally isometric to the euclidean plane.  Manifolds arising as quotients of affine space by discrete subgroups of $\text{Isom}(n,\mathbb{R})$  are known as euclidean manifolds.  Bieberach showed that such manifolds are finite coverings of euclidean tori.  
\end{example}
This construction of finding discrete subgroups of the affine group that act properly and freely on affine space yield an entire class of affine manifolds.  Invariants of the group such as vector fields, covector fields, and metrics descend to the quotient and are studied extensively in the field of geometric structure.  
%
%
%
%
\begin{example}\label{ex2}
Let $D$ be a cyclic subgroup of $\text{Aff}(2,\mathbb{R})$ which acts on the punctured plane $\mathbb{C}^{\times}$ by positive dilations centered at the origin.  Say that $D$ is generated by some $\lambda > 0$, so $D$ acts properly and freely on $\mathbb{C}^{\times}$ as in the previous example, and thus the quotient $N = \mathbb{C}^{\times}/D$ inherits an affine structure.  The group $D$ preserves the lorrentzian metric $m = (dx^{2}+dy^{2})/(x^{2}+y^{2})$ and thus descends to a lorrentzian metric on $N$.  In contrast to the previous riemannian example, this metric is incomplete.  In fact, geodesics on $\mathbb{C}^{\times}$ pointed towards the origin will descend to geodesics on $N$ that become undefined in finite time.  This should be contrasted with the Hopf-Rinow of riemannian geometry which states that every closed riemannian manifold is geodesically complete.  
\end{example}
%
%
%
%
\begin{example}\label{ex3}
One may generalize the construction in Example \ref{ex1} in the following fashion.  Pick a quadrilateral $Q$ in the affine plane.  From the bottom left vertex and reading counterclockwise, label the edges $\alpha,\beta,\gamma$ and $\delta$.  Construct a new quadrilateral in the following fashion.  Rotate and scale $Q$ at the bottom left vertex in such a fashion that the rotated $\delta$ differs from $\beta$ by translation along $\alpha$.  Translate the result along $\alpha$ to yield a new quadrilateral $Q'$ with edges $\alpha',\beta',\gamma',$ and $\delta'$.  This process glues $\beta$ of $Q$ to $\delta'$ of $Q'$.  Repeat this process with each pair of edges indefinitely.  Figure \ref{fig:1} above shows an example of this process where the edges are no longer simple translates of each other.
%
%
%
%
\begin{figure}
\begin{center}
\includegraphics[scale=0.35]{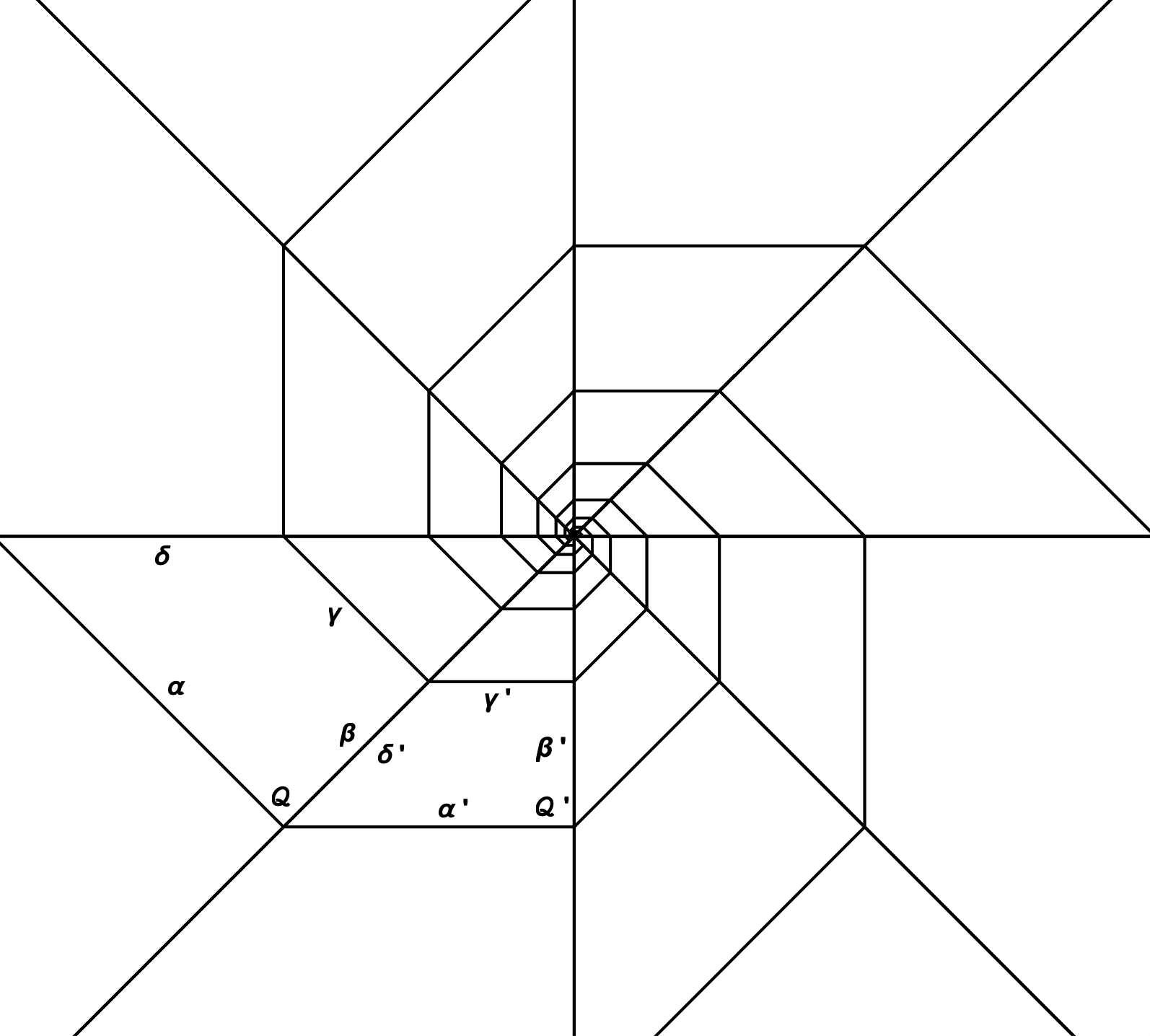}
\end{center}
\caption{The quadrilaterals $Q$ and $Q'$ are labeled with their corresponding edges in a counterclockwise fashion.  Note the edges $\beta$ and $\delta'$ are identified in the quotient.  The quadrilateral $Q$ serves as a fundamental domain for the corresponding group action.  [Made in Mathematica 12 Student Version] }
\label{fig:1}
\end{figure}
The case where both pairs of opposite edges of $Q$ are parallel is the method of identification as in Example \ref{ex1} and defines a euclidean structure on the torus.  In the case where pairs of opposite edges of $Q$ are not necessarily parallel, this method provides an inequivalent structure on the torus known as a similarity structure which are affine structures whose holonomy lies within the group of similarity transforms of affine space.  In fact, Figure \ref{fig:1} was generated by the quadrilateral $Q = \{(0,0),(2,0),(1,1),(0,1)\}$ and the group of similarity transformations generated by
\[
a = \left(\begin{array}{cc}
1/2 & 0 \\
0 & 1/2 \end{array}
\right) 
\left(\begin{array}{c}
0\\
1
\end{array}
\right)
 \text{ and }
 b = \left(\begin{array}{cc}
1 & -1 \\
1 & 1 \end{array}
\right) 
\left(\begin{array}{c}
2\\
0
\end{array}
\right)
\] 
\end{example}
%
%
These three examples serve to illustrate the complexities that arise once one departs from the riemannian case to the affine case.  For those interested in learning more about the affine structures supported by the two torus, Oliver Baues provides an excellent treatment about different types of affine structures supported on the two torus \cite{bib1}.  \\
\\
%
%
%
%
Given an affine structure on a manifold there is a natural associated local diffeomorphism from the universal cover of $M$ to affine space and a homomorphism from $\pi_{1}(M,p) \longrightarrow \text{Aff}(n,\mathbb{R})$.  The local diffeomorphism is called the developing map, the group homomorphism is called the holonomy, and the two together are called a developing pair.  A brief description of these maps is provided below, but further details about their construction may be found in William Goldman's Geometric Structure on Manifolds \cite{bib2}.\\
\\
Base the fundamental group at a point $p \in M$.  Let $p \in (U,\Phi)$ be an affine coordinate patch about $p$.  Each path $\gamma : [0,1] \longrightarrow M$ beginning at $p$ may be assigned a point in affine space in the following fashion.  Cover the path $\gamma$ by $(k+1)$-coordinate patches beginning with $(U,\Phi)$.  Label these patches by $(U_{i},\Phi_{i})$ where $(U_{0},\Phi_{0}) = (U,\Phi)$.  Pick a mesh of times $0=t_{0} <  t_{1} <  \hdots < t_{k} < t_{k+1} = 1$ in $[0,1]$ so that each $\gamma(t_{i}) \in U_{i-1}\cap U_{i}$ for $i = 1\hdots k$.  Let $\gamma_{i}$ be the restriction of $\gamma$ to $[t_{i},t_{i+1}]$ for each $i = 0 \hdots k$.\\
\\
Inductively define paths in affine space in the following fashion.  Let $\alpha_{0} = \Phi_{0}\circ \gamma_{0}$.  Let $V_{1}$ be the connected component of $U_{0}\cap U_{1}$ containing $\gamma(t_{1})$ and $g_{0,1}$ be the affine automorphism so that $g_{0,1}|_{V_{1}} = \Phi_{0}\circ\Phi_{1}^{-1} : \Phi_{1}(V_{1}) \longrightarrow \Phi_{0}(V_{1})$.  Define $\alpha_{1} = g_{0,1}(\Phi_{1}\circ\gamma_{1})$.  Note that the initial point of $\alpha_{1}$ is the terminal point of $\alpha_{0}$.  Let $V_{2}$ be the connected component of $U_{1}\cap U_{2}$ containing $\gamma(t_{2})$ and $g_{1,2}$ be the affine automorphism so that $g_{1,2}|_{V_{2}} = \Phi_{1}\circ \Phi_{2}^{-1} : \Phi_{2}(V_{2}) \longrightarrow \Phi_{1}(V_{2})$.  Define $\alpha_{2} = g_{0,1}g_{1,2}(\Phi_{2}\circ\gamma_{2})$.  Note the initial point of $\alpha_{2}$ is the terminal point of $\alpha_{1}$.  Continue this process inductively to obtain $(k+1)$-paths into affine space and concatenate them to obtain the path
%
%
%
%
\begin{align}\label{eq2}
\alpha_{0}\cdot\alpha_{1}\cdot\hdots\cdot\alpha_{k} = (\Phi_{0}\circ &\gamma_{0})\cdot(g_{0,1}(\Phi_{1}\circ \gamma_{1}))\cdot (g_{0,1}g_{1,2}(\Phi_{2}\circ\gamma_{2}))\cdot\hdots \nonumber \\
&\cdot(g_{0,1}g_{1,2}\hdots g_{k-1,k}(\Phi_{k}\circ \gamma_{k})) 
\end{align}
The developing map is defined as the terminal point of this path.  Figure \ref{fig:2} illustrates this construction with three charts.
%
%
%
\begin{figure}
\begin{center}
\includegraphics[scale=0.25]{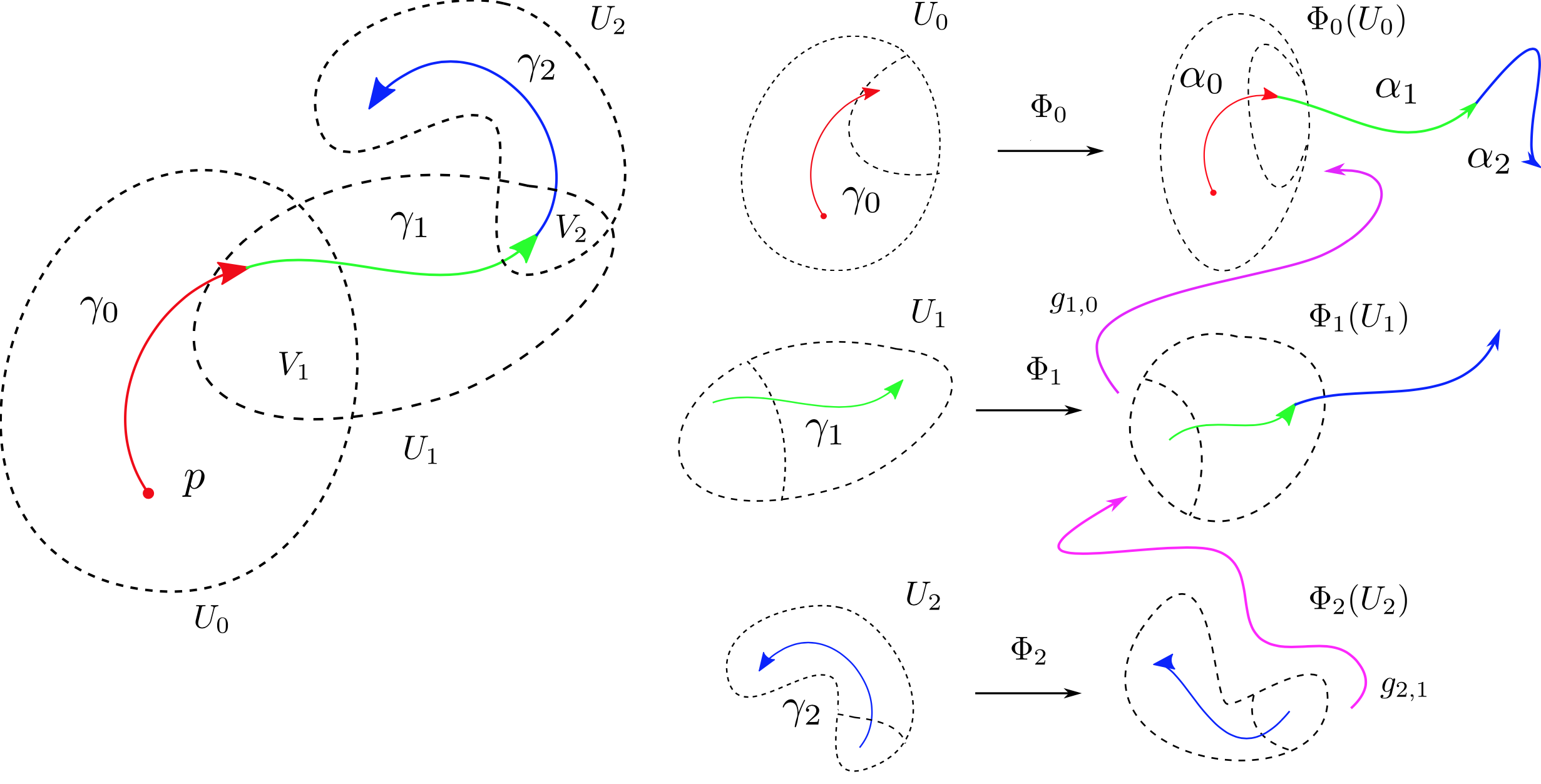}
\end{center}
\caption{Three charts $U_{0}, U_{1}$, and $U_{2}$ cover a path $\gamma$ based at $p$.  The path is separated into three pieces $\gamma_{0}, \gamma_{1}$ and $\gamma_{2}$ in red, green, and blue respectively.  The charts are shown to the right with their images in affine space.  The pink arrows represent the affine transformations taking one affine image to the next, i.e. $g_{0,1}$ and $g_{1,2}$.  For example, the blue path in the third patch is mapped to the to blue path in the second patch by $g_{1,2}$.  This construction yields the concatenation $\alpha_{1}\cdot\alpha_{2}\cdot\alpha_{3}$ in the top right whose terminal point is the developing map.}
\label{fig:2}
\end{figure}
Several facts need to be verified about this assignment.  One must show that this map is independent of the choice of charts after the initial chart about $p\in(U,\Phi)$ is chosen.  In addition, one must show that the map is well-defined up to homotopy of paths based at $p$.  After these technical details are established, this assignment induces a local diffeomorphism from the universal cover of $M$ based at $p$ to affine space which is denoted $\text{dev} : \widetilde{M} \longrightarrow \mathbb{A}^{n}$.  \\
\\
Let $\gamma$ be a path based at $p$ contained in a chart $(U,\Phi)$ as above, and let $[\beta] \in \pi_{1}(M,p)$.  As the developing map is defined in terms of homotopy classes of paths based at $p$, it is natural to consider how the developing map behaves by precomposition of loops based at $p$.  That is, one may consider how Equation \ref{eq2} changes by considering the path $\beta\cdot\gamma$ where $\beta$ is a representative of the homotopy class $[\beta]$.  As $\beta$ begins and ends at $p$, one may take the initial and terminal charts covering $\beta$ to both be $(U,\Phi)$ with charts $(V_{i},\Theta_{i})$ covering the remainder of $\beta$.  Say that the corresponding construction applied to $\beta$ yields paths $\delta_{1},\delta_{2},\hdots, \delta_{j}$ in affine space with change of coordinate elements $h_{0,1},h_{1,2},\hdots, h_{j-1,j}$.  Then the developing map applied to the concatenation $\beta\cdot \gamma$ yields
%
%
%
\begin{align*}\label{eq3}
&\delta_{1}\cdot \hdots \cdot \delta_{j} \cdot \gamma_{1}\cdot \hdots \cdot \gamma_{k} = \nonumber \\
&\phantom{==} (\Phi_{0}\circ \beta_{0})\cdot(h_{0,1}(\Theta_{1}\circ \beta_{1}))\cdot (h_{0,1}h_{1,2}(\Theta_{2}\circ\beta_{2}))\cdot\hdots \nonumber\\
&\phantom{====}\cdot(h_{0,1}h_{1,2}\hdots h_{j-1,j}(\Theta_{j}\circ \beta_{j}))\cdot \gamma_{1}\cdot \hdots \cdot \gamma_{k} \nonumber
\end{align*}
As the right hand side product of the $\gamma_{i}$'s is left unchanged, and the developing map is defined as the terminal point of the constructed path above, one can see that
%
%
%
\begin{equation}\label{eq4}
\left(\delta_{1}\cdot \hdots \cdot \delta_{j} \cdot \gamma_{1}\cdot \hdots \cdot \gamma_{k}\right)(1) = h_{0,1}h_{1,2}\hdots h_{j-1,j}\left(\gamma_{1}\cdot\hdots\cdot\gamma_{k}\right)(1)
\end{equation}
Since the path $\gamma$ was arbitrary, Equation \ref{eq4} holds for all such paths based at $p$, so precomposition with an element of $[\beta]\in \pi_{1}(M,p)$ yields a difference in the developing map by the element of the affine group $h_{0,1}h_{1,2}\hdots h_{j-1,j}$ which corresponds to $[\beta]$.  This element is known as the holonomy of $[\beta]$, denoted $\text{hol}[\beta]$, and defines a homomorphism, known as the holonomy map, from $\pi_{1}(M,p) \longrightarrow \text{Aff}(n,\mathbb{R})$.  Equation \ref{eq4} is the statement that the developing map is equivariant with respect to the holonomy homomorphism in the sense of the following commutative diagram.
%
%
%
\begin{equation}\label{eq5}
	\begin{tikzcd}
		\widetilde{M} \arrow{r}{[\beta]} \arrow{d}[swap]{\text{dev}} & \widetilde{M} \arrow{d}{\text{dev}}\\
		\mathbb{A}^{n} \arrow{r}[swap]{\text{hol}[\beta]}& \mathbb{A}^{n}
	\end{tikzcd}
\end{equation}
%
%
%
The pair $(\text{dev},\text{hol})$ is known as a developing pair for the affine structure on $M$.  The construction of the developing map above carries over to the broader context of $(G,X)$-structures on manifolds wherein one assumes that a lie group $G$ acts strongly effectively on a manifold $X$.  A very nice exposition about $(G,X)$-structures may be found in Stephan Schmitt's Geometric Manifolds \cite{bib4} whereas \cite{bib2} provides a very thorough general reference.


\section{Affine Structures with an Invariant Line}
\label{sec:2}
The purpose of this section is establish more specialized notation, preliminary observations about the consequences of having an affine structure whose affine holonomy admits an invariant affine line, and to prove a theorem about the non-existence of certain affine structures.  \\
\\
Let $l \subset \mathbb{A}^{n+1}$ be a line and $G \leq \text{Aff}(n+1,\mathbb{R})$ be the group of affine automorphisms preserving $l$.  Pick an origin on $l$ and identify $\mathbb{A}^{n+1}$ with $\mathbb{R}^{n+1}$.  Rotate about the origin so that $l$ aligns with the $x$-axis in $\mathbb{R}^{n+1}$ and $l$ identifies with the first factor of $\mathbb{R}$ in $\mathbb{R}\times\mathbb{R}^{n}$.  Up to conjugation $G$ is isomorphic to 
%
%
%
%
\begin{equation}\label{eq6}
G =
\left\{
	\left(
		\begin{array}{cc}
		r & w \\
		0 & A
		\end{array}
	\right)
	\left(
		\begin{array}{c}
		d \\
		0
		\end{array}
	\right) 
	\Bigg|\,
	r \neq 0, d\in\mathbb{R}, w^{T} \in \mathbb{R}^{n}, A \in \text{GL}(n,\mathbb{R})
\right\}
\end{equation}
For this purpose of this paper the coordinates $x$ and $y^{1},\hdots, y^{n}$ are reserved for the first and second factors of $\mathbb{R}\times\mathbb{R}^{n}$ respectively.  In an abuse of notation, $\mathbb{R}$ will frequently denote the invariant line $\mathbb{R}\times0 \subset \mathbb{R}\times\mathbb{R}^{n}$.   The instances in which this occurs will be clear throughout the paper.  Before stating one of the theorems of this paper, a definition is in order.   
%
%
%
%
%
\begin{definition}\label{def2}
Let $M$ be a closed affine manifold with a developing pair $(\text{dev},\text{hol})$.  If the developing map is a covering onto $\mathbb{A}^{n}$, then $M$ is called complete.
\end{definition}
%
%
%
Since both the universal cover of $M$ and $\mathbb{A}^{n}$ are simply connected, this definition implies that the developing map is a diffeomorphism onto $\mathbb{A}^{n}$.  In this case, Corollary \ref{cor3} in Section \ref{sec:7} yields that the holonomy group, $\text{hol}(\Gamma)$, acts both properly and freely on affine space, and $M$ is diffeomorphic to $\mathbb{A}^{n}/\text{hol}(\Gamma)$.  In this context the manifold $M$ can be recovered from the image of the holonomy homomorphism.  It is a non-trivial problem to construct inequivalent structures on a manifold with the same holonomy groups.  The interested reader is referred to Projective Structure with Fuchsian Holonomy \cite{bib5}.  
%
%
%
\begin{theorem}\label{thm1}
For $n \geq 1$, there are no complete affine structures on closed $(n+1)$-manifolds whose holonomy lies in the group $G$ of line preserving affine automorphisms.   
\end{theorem}
This proof is broken into two pieces.  Let $H \leq G$ denote a subgroup of $G$ as in Equation \ref{eq6} that acts both properly and freely on $\mathbb{R}\times\mathbb{R}^{n}$ with a compact quotient.  The proof begins by showing that $H$ must act purely by translations on the invariant line, otherwise, there will be an accumulation point of the group action on the invariant line, contradicting properness.  After that is established, $H$ will be shown to be cylic, whereby the quotient manifold will be shown to be a mapping torus, $M_{A}$, of a linear map $A : \mathbb{R}^{n} \longrightarrow \mathbb{R}^{n}$.  Certain topological obstructions will prevent this from occurring and yield a contradiction.  
%
%
%
\begin{proof}
Let $H \leq G$ act both properly and freely on $\mathbb{R}\times\mathbb{R}^{n}$ with a compact quotient.  Without loss of generality the scaling factor, $r \neq 0$, as in Equation \ref{eq6} may be taken to be non-negative.  The subgroup of $G^{+}\leq G$ whose scale factors $r$ on the invariant line are positive form an index two subgroup of $G$.  Consequently the quotient $\mathbb{R}\times\mathbb{R}^{n}/G^{+}$ is a double cover of the quotient by $G$, and thus preserves both an affine structure and compactness.  In addition one may lift to the orientable double cover to assume the manifold $\mathbb{R}\times\mathbb{R}^{n}/G^{+}$ is orientable.    \\
\\
If there is indeed an element $h \in H$ so that $r \neq 1$, then there is a solution to the equation $rx + d = x$.  Conjugate by translation along the invariant line to assume $h$ is of the form
%
%
%
\begin{equation}\label{eq7}
h = 
\left(
	\begin{array}{cc}
	r & w \\
	0 & A
	\end{array}
\right)
\left(
	\begin{array}{c}
	0 \\
	0
	\end{array}
\right) \nonumber
\end{equation}
Note that $h$ acts linearly on $\mathbb{R}\times\mathbb{R}^{n}$ and fixes the origin.  That said, the orbit of the origin and say for example $(1,0) \in \mathbb{R}\times\mathbb{R}^{n}$ are inseparable by open sets.  If $(1,0)$ in indeed in the orbit of the origin, one may pick a point arbitrarily close to $(1,0)$ on the invariant line that is not.  The cyclic group generated by $h$ fixes the origin, whereas $h^{n}(1,0)$ will tend arbitrarily close to the origin along the invariant line for sufficiently large positive or negative values of $n$ depending on whether $r$ is less than or bigger than one.  The orbits of the origin and $(1,0)$ are inseparable so the action is not proper, and thus elements that fail to act by pure translation on the invariant line are not in $H$.  \\
\\
Now assume that $H$ acts by pure translations on the invariant line $\mathbb{R}$.  Since the action is proper, it follows that map $T : H \longrightarrow \mathbb{R}$ defined by
%
%
%
\begin{equation}\label{eq8}
T\left(\left(
	\begin{array}{cc}
	1 & w \\
	0 & A
	\end{array}
\right)
\left(
	\begin{array}{c}
	d \\
	0
	\end{array}
\right) \right)= d \nonumber
\end{equation}
is a homomorphism.  By properness of the action of $H$, the image of $T$ is cyclic, as a dense image would yield an accumulation point on the invariant line.  Since the action of $H$ on $\mathbb{R}\times\mathbb{R}^{n}$ is free, this homomorphism is injective.  For if two elements $h,h' \in H$ yield the same translational part, then $h^{-1}h'$ fixes the origin in $\mathbb{R}\times\mathbb{R}^{n}$, and by freeness $h = h'$.  Thus $H$ is generated by some 
\[
h = \left(
	\begin{array}{cc}
	1 & w \\
	0 & A
	\end{array}
\right)
\left(
	\begin{array}{c}
	d \\
	0
	\end{array}
\right) \text{ where } d\in\mathbb{R}, \, w^{T} \in \mathbb{R}^{n} \text{ and } A \in \text{GL}(n,\mathbb{R})
\]
Note if $w = 0$, then $\mathbb{R}\times\mathbb{R}^{n}$ quotiented by the cyclic group $H$ is homeomorphic to the mapping cylinder of the linear map $A:\mathbb{R}^{n} \longrightarrow \mathbb{R}^{n}$ as claimed.  If $w\neq 0$, then conjugating $h$ by a sheer along $\mathbb{R}$ yields
%
%
%
\begin{equation}\label{eq9}
\left(
	\begin{array}{cc}
	1 & v \\
	0 & I_{n}
	\end{array}
\right)
\left(
	\begin{array}{cc}
	1 & w \\
	0 & A
	\end{array}
\right)
\left(
	\begin{array}{c}
	d \\
	0
	\end{array}
\right)
\left(
	\begin{array}{cc}
	1 & -v \\
	0 & I_{n}
	\end{array}
\right) = 
\left(
	\begin{array}{ccc}
	1 & \phantom{,}& w+v(A-I_{n}) \\
	0 &  \phantom{,}& A
	\end{array}
\right)
\left(
	\begin{array}{c}
	d \\
	0
	\end{array}
\right)
\end{equation}
If $(A-I_{n})$ is invertible then $w + v(A-I_{n})$ can be chosen to be zero, thus providing the desired homeomorphism.  Such a choice of $v$ is available if $(A-I_{n})$ is invertible, which is to say that $\lambda = 1$ is not an eigenvalue of $A$. \\
\\
Assume $w \neq 0$ and $\lambda = 1$ is an eigenvalue of $A$.  If so, then there is a $u \in \mathbb{R}^{n}$ for which $Au = u$.  Moreover,  $w^{T}$ and $u$ are necessarily perpendicular.  For if not let $k \in \mathbb{R}$, and then
%
%
%
\begin{equation}\label{eq10}
 \left(
	\begin{array}{cc}
	1 & w \\
	0 & A
	\end{array}
\right)
\left(
	\begin{array}{c}
	d \\
	0
	\end{array}
\right) 
\left(
	\begin{array}{c}
	0 \\
	ku
	\end{array}
\right) =
\left(
	\begin{array}{c}
	0+k(wu)+d \\
	ku
	\end{array}
\right) \nonumber
\end{equation}
where $A(ku) = ku$ as $u$ is an eigenvector of $A$ with eigenvalue $1$. Since $wu \neq 0$ as $w^{T}$ and $u$ are not perpendicular, there is a choice of $k$ for which $k(wu) + d  = 0$ and thus $(0,ku)$ is fixed by a generator of $H$ contradicting the fact that $H$ acts freely on $\mathbb{R}\times\mathbb{R}^{n}$.  Thus, $w^{T}$ and $u$ are perpendicular as claimed.\\
\\
Let $U$ denote the plane in $\mathbb{R}\times\mathbb{R}^{n}$ spanned by $(1,0)$ and $(0,u)$.  Since $w^{T}$ and $u$ are orthogonal, $U$ is a closed subspace invariant under the action of $H$, so its quotient $U/H$ is a compact submanifold of $\mathbb{R}\times\mathbb{R}^{n}/H$.  This though is a contradiction as $U/H$ is diffeomorphic to $S^{1}\times\mathbb{R}$, and is therefore non-compact.   Since $w$ was assumed to be non-zero, this necessitates $\lambda = 1$ is not an eigenvalue of $h$.\\
\\
Since $\lambda = 1$ is not an eigenvalue of $h$, one may conjugate by translation along the invariant line as in Equation \ref{eq9} to assume $h$ is a matrix of the form
\[
h =  \left(
	\begin{array}{cc}
	1 & 0 \\
	0 & A
	\end{array}
\right)
\left(
	\begin{array}{c}
	d \\
	0
	\end{array}
\right) 
\]
%
%
%
Thus the quotient $\mathbb{R}\times\mathbb{R}^{n}/H$ is homeomorphic to $M_{A}$, the mapping torus of the linear map $A : \mathbb{R}^{n} \longrightarrow \mathbb{R}^{n}$.  A standard result in topology \cite{bib6} provides the long exact sequence of homology groups of a mapping torus is given by
%
%
%
%
%
%
\begin{equation}\label{eq11}
\hdots \longrightarrow H_{n+1}(\mathbb{R}^{n}) \longrightarrow H_{n+1}(M_{A}) \longrightarrow H_{n}(\mathbb{R}^{n}) \longrightarrow \hdots \nonumber
\end{equation}
Since $\mathbb{R}^{n}$ is contractible and $n \geq 1$, this necessitates that $H_{n+1}(M_{A})$ is trivial.  Since $M_{A}$ is an $(n+1)$-dimensional, compact, oriented manifold, its top homology is necessarily non-trivial.  This shows that there are no complete structures on a compact $(n+1)$-dimensional affine manifold whose affine holonomy preserves an affine line.
\end{proof}
It is worth noting that although Theorem \ref{thm1} forbids the existence of a complete structure on a closed manifold $M$ whose holonomy preserves an invariant line, this theorem says nothing about the existence of non-complete structures.  In fact, there are plenty of examples of non-complete structures in which the developing map fails to be a covering onto all of $\mathbb{R}\times\mathbb{R}^{n}$.  
%
%
%
\begin{example}\label{ex4}
Pick a $\lambda > 1$ and let $M$ be $\mathbb{R}\times \mathbb{C}^{\times}/H$ where $H$ is the subgroup of affine transformations generated by
%
%
%
\begin{equation}\label{eq12}
a = 
\left(
	\begin{array}{ccc}
	1 & 0 & 0 \\
	0 & 1 & 0 \\
	0 & 0 & 1
	\end{array}
\right)
\left(
	\begin{array}{c}
	1\\
	0\\
	0
	\end{array}
\right) \text{ and }
b = 
\left(
	\begin{array}{ccc}
	1 & 0 & 0 \\
	0 & \lambda & 0 \\
	0 & 0 & \lambda
	\end{array} \nonumber
\right)
\end{equation}
$M$ defines an affine structure on the three-torus whose affine holonomy preserves the invariant line defined by the $x$-axis.  In this case the induced action on the invariant line $\mathbb{R}$ is purely translational, yet the invariant line lies entirely outside the developing image of this affine structure.  
\end{example}
The fact that the invariant line lies outside the developing image is no coincidence in the case where the affine holonomy acts purely by translations on the invariant line.  The remainder of this paper is dedicated to showing this is always the case.  \\
\\
The basic strategy is to show that if the affine holonomy does indeed act by translations on the invariant line and the developing image meets the invariant line, the affine structure is complete thus yielding a contradiction to Theorem \ref{thm1}.  To show the developing map is a diffeomorphism, two main techniques will be employed.  \\
\\
The first is show that if the affine manifold admits a parallel flow, one can construct `large' open submanifolds of the universal cover upon which the restricted developing map is a diffeomorphism.  The second is to show that if the manifold admits a so called `cylindrical' flow, these `large' open sets can be taken to be arbitrarily large.  Once these two facts are established, the proof follows immediately as a consequence of Theorem \ref{thm1}.


\section{Parallel Flow}
\label{sec:3}
Let $M$ be a closed affine $(n+1)$-dimensional manifold with a developing pair $(\text{dev},\text{hol})$ and fundamental group $\Gamma = \pi_{1}(M,p)$ acting on the universal cover $\widetilde{M}$ satisfying Equation \ref{eq5}.  Assume the affine holonomy group, $H = \text{hol}(\Gamma) \leq \text{Aff}(n+1,\mathbb{R})$, preserves a parallel vector field $V$.  Pick an origin in $\mathbb{A}^{n+1}$, and identify $\mathbb{A}^{n+1}$ with $\mathbb{R}\times\mathbb{R}^{n}$ and $\text{Aff}(n+1,\mathbb{R})$ with $\text{GL}(n+1,\mathbb{R})\ltimes \mathbb{R}^{n+1}$.  By the natural identification of parallel vector fields $V$ on $\mathbb{R}\times\mathbb{R}^{n}$ with the tangent space of $\mathbb{R}\times\mathbb{R}^{n}$ at a point, the statement that the affine holonomy preserves a parallel vector field is equivalent to the statement that there exists a $v \in T_{0}(\mathbb{R}\times\mathbb{R}^{n})$ so that for each $h \in H$, $v$ is an eigenvector of the linear part of $h$.  Thus, up to conjugation, one may assume the affine holonomy lies inside the subgroup of affine automorphisms given by    
%
%
%
\begin{equation}\label{eq13}
P=  \left\{
	\left(
		\begin{array}{ccc}
		1 & w\\
		0 & A
		\end{array}
	\right)
	\left(
		\begin{array}{c}
		d \\
		v
		\end{array}
	\right)  \,  \Bigg| \, d\in\mathbb{R}, w^{T}, v \in \mathbb{R}^{n}, A\in \text{GL}(n,\mathbb{R}) 
\right\}
\end{equation}
The non-vanishing parallel $H$-invariant vector field $\partial/\partial x$ lifts to a non-vanishing parallel $\Gamma$-invariant vector field $\widetilde{V}$ which descends to a non-vanishing parallel vector field $V$ on $M$.  As $M$ is compact, the flow associated to $V$ is complete, and so too is the flow associated to $\widetilde{V}$.  Denote this flow on $\widetilde{M}$ by $\widetilde{\Theta}_{t}$.  As $\widetilde{V}$ is related to $\partial/\partial x$ by the developing map, so too are their corresponding flows.  Denoting $T_{t}$ as the translational flow corresponding to $\partial/\partial x$, one obtains the commutative diagram
%
%
%
\begin{equation}\label{eq14}
	\begin{tikzcd}
		&\widetilde{M} \arrow{r}{\widetilde{\Theta}_{t}} \arrow{d}[swap]{\text{dev}} &\widetilde{M} \arrow{d}{\text{dev}}\\
		&\mathbb{R}\times\mathbb{R}^{n} \arrow{r}[swap]{T_{t}}&\mathbb{R}\times\mathbb{R}^{n}
	\end{tikzcd} \nonumber
\end{equation}
As $T_{t}$ is the flow associated the vector field $\partial/\partial x$ which is invariant under the holonomy, this necessitates that $T_{t}$ commutes with each element of the holonomy as the holonomy will take flow lines to flow lines.  The same statement holds for the commutativity of $\widetilde{T}_{t}$ and $\Gamma$.  \\
\\
It is clear that the $\mathbb{R}$-action on $\mathbb{R}\times\mathbb{R}^{n}$ given by translation $T_{t}$ is both free and proper, consequently, so too is the $\mathbb{R}$-action on $\widetilde{M}$ by Lemma \ref{lm2}.  Thus $\widetilde{M}$ is a principal $\mathbb{R}$-bundle over the quotient manifold $\widetilde{M}/\mathbb{R}$.  Denote this quotient manifold by $N$ and the associated quotient map by $q : \widetilde{M} \longrightarrow N$.  Since $\mathbb{R}$ is contractible, the principal $\mathbb{R}$-bundle structure on $\widetilde{M}$ is isomorphic to the trivial bundle $p_{2} : \mathbb{R}\times N \longrightarrow N$ where $p_{2}$ is factor projection onto the second factor.  By triviality of the principal bundle, there is an $\mathbb{R}$-equivariant diffeomorphism $\Phi : \widetilde{M} \longrightarrow \mathbb{R}\times N$ for which the below diagrams commute.
%
%
%
\begin{equation}\label{eq15}
	\begin{tikzcd}
	\widetilde{M} \arrow{dr}[swap]{q} \arrow{rr}{\Phi} & 		&  \arrow{dl}{p_{2}} \mathbb{R}\times N  \\
									& N  \\
	\end{tikzcd}
	\text{ and }
	\begin{tikzcd}
		&\widetilde{M} \arrow{r}{\widetilde{\Theta}_{t}} \arrow{d}[swap]{\Phi} &\widetilde{M} \arrow{d}{\Phi}\\
		&\mathbb{R}\times N \arrow{r}[swap]{\widetilde{T}_{t}}&\mathbb{R}\times N
	\end{tikzcd}
\end{equation}
%
%
%
A standard result in the theory principal bundles states that a principal bundle is trivial if and only if the bundle admits a global section, and thus $N$ may be thought of as an embedded submanifold of $N \subset \widetilde{M}$ for which the saturation of $N$ by the $\mathbb{R}$-action yields all of $\widetilde{M}$.  \\
\\
The identification of $\Phi : \widetilde{M} \longrightarrow \mathbb{R}\times N$ provides a $\Gamma$-action on $\mathbb{R}\times N$ via conjugation by $\Phi$.  Specifically, $[\gamma](t,n) = (\Phi\circ [\gamma] \circ \Phi^{-1})(t,n)$.  The $\Gamma$-action on $\mathbb{R}\times N$ commutes with the $\mathbb{R}$-action on $\mathbb{R}\times N$ as per consequence of the definition of the $\Gamma$-action on $\mathbb{R}\times N$ and Equation \ref{eq15}.  In addition, one obtains the commutative diagram
%
%
%
\begin{equation}\label{eq16}
	\begin{tikzcd}
	 \mathbb{R} \times N \arrow{r}{[\gamma]} \arrow{d}[swap]{\Phi^{-1}} &  \mathbb{R} \times N \arrow{d}{\Phi^{-1}} \\
	\widetilde{M} \arrow{r}{[\gamma]} \arrow{d}[swap]{\text{dev}}& \widetilde{M} \arrow{d}{\text{dev}} \\
	\mathbb{R}\times\mathbb{R}^{n} \arrow{r}[swap]{\text{hol}[\gamma]} & \mathbb{R}\times\mathbb{R}^{n}
	\end{tikzcd}
\end{equation}
The composition of $\text{dev}\circ \Phi^{-1}$ provides a local diffeomorphism of $\mathbb{R}\times N$ into $\mathbb{R}\times\mathbb{R}^{n}$ which is equivariant with respect to the $\Gamma$-action on $\mathbb{R}\times N$.  In the standard abuse of notation, this composition is also denoted $\text{dev} : \mathbb{R}\times N \longrightarrow \mathbb{R}\times\mathbb{R}^{n}$.  \\
\\
By the $\mathbb{R}$-equivariance of $\Phi$ as in Equation \ref{eq15}, one obtains the commutative diagram
%
%
%
\begin{equation}\label{eq17}
	\begin{tikzcd}
		&\mathbb{R}\times N \arrow{d}[swap]{\text{dev}} \arrow{r}{\widetilde{T}_{t}} & \mathbb{R}\times N\arrow{d}{\text{dev}}  \\
		&\mathbb{R}\times\mathbb{R}^{n} \arrow{r}[swap]{T_{t}} & \mathbb{R}\times\mathbb{R}^{n}
	\end{tikzcd}
\end{equation}
Before continuing, it is worth noting that Equation \ref{eq17} is more or less the statement that there exists a transverse submanifold $N$ that generates the $\mathbb{R}$-action on the universal cover, $\mathbb{R}\times N$, in such a fashion that that the developing map takes flow lines to flow lines.  As both $\mathbb{R}$-actions are free and proper, one may pass the vertices of Equation \ref{eq17} to their quotients.  As both the $\Gamma$-action on $\mathbb{R}\times N$ and the holonomy on $\mathbb{R}\times\mathbb{R}^{n}$ commute with their respective $\mathbb{R}$-actions, the actions of $\Gamma$ and the holonomy group pass to their quotients, which are also abusively denoted by $[\gamma]$ and $\text{hol}[\gamma]$.  Specifically, one has the commutative square
%
%
%
\begin{equation}\label{eq18}
	\begin{tikzcd}
	        & N \arrow{d}[swap]{\overline{\text{dev}}} \arrow{r}{[\gamma]} & N\arrow{d}{\overline{\text{dev}}}  \\
		&\mathbb{R}^{n} \arrow{r}[swap]{\text{hol}[\gamma]} & \mathbb{R}^{n}
	\end{tikzcd}
\end{equation}
where $\overline{\text{dev}} : N \longrightarrow \mathbb{R}^{n}$ is a local diffeomorphism, and the holonomy $\text{hol}[\gamma]$ is acting affinely on $\mathbb{R}^{n}$ by the affine action induced by the matrices and vectors $A \in \text{GL}(n,\mathbb{R})$ and $v\in\mathbb{R}^{n}$ as in Equation \ref{eq13}.  \\
\\
Equation \ref{eq18} looks deceptively as though it defines an affine structure on $N/\Gamma$.  This though assumes that the induced action of $\Gamma$ on $N$ is both free and proper, which is not necessarily the case.  The following example below illustrates possible obstructions when passing the quotient.
%
%
%
%
%
\begin{example}\label{ex5}
Let $D$ be a closed unit disk centered at the origin in $\mathbb{R}^{2}$.  Let $\Gamma$ be the cyclic group acting on $\mathbb{R}\times D$ generated by translation along $\mathbb{R}$ and rotation by an irrational angle $\theta$.  For example let $\Gamma$ be generated by
\[
a = \left(
\begin{array}{ccc}
1 & 0 & 0\\
0 & \cos(\theta) & -\sin(\theta) \\
0 & \sin(\theta) & \cos(\theta)
\end{array}
\right)
\left(
\begin{array}{c}
1\\
0\\
0
\end{array}
\right)
\]
\end{example}
Let $M = \mathbb{R}\times D/\Gamma$ which is a compact three-dimensional manifold with boundary as the action of $\Gamma$ on $\mathbb{R}\times D$ is both free and proper.  In fact, $M$ is the mapping torus of the map $a : D \longrightarrow D$.  Since the holonomy $\Gamma$ preserves the vector field $\partial/\partial x$, $\Gamma$ maps flow lines to flow lines whose images are $\mathbb{R}\times p$ for each $p \in D$.  \\
\\
The induced action of $\Gamma$ on $D$ is neither free no proper.  In particular the induced action of $a$ on the unit disk preserves the flow line of the origin, and is consequently not free.  In addition, the orbit of each point $p \in D$ with a radius $0 < r \leq 1 $ has an orbit that is dense on the corresponding circle of radius $r$, and consequently the induced action of $\Gamma$ on $D$ is not proper.  Though this is example with boundary, it nevertheless conveys the fragility of proper actions.  Proper actions in general do not pass to proper actions on quotients.  This should be contrasted with case where a proper action is lifted, see for example Lemma \ref{lm2}.  \\
\\
As mentioned towards the end of Section \ref{sec:2}, the goal of this section is to prove the existence of `large' open subsets upon which the developing map when restricted to them are diffeomorphisms.  As all the necessary terminology is in place, the theorem may be stated.
%
%
%
\begin{theorem}\label{thm2}
Let $M$ be a closed affine $(n+1)$-dimensional manifold whose affine holonomy admits a parallel vector field.  There exists a complete parallel flow on the universal cover of $M$ equivariant with respect to the parallel flow induced by the parallel vector field on affine space.  Additionally, there exists open subsets of the universal cover invariant under the parallel flow for which the restricted developing map is a diffeomorphism onto its image.  
\end{theorem}
%
%
%
\begin{proof}
The previous paragraphs establish the existence of a complete parallel flow satisfying the equivariance condition in the statement of the theorem.  To finish the proof of Theorem 2, it suffices to show the existence of the open subset of $\mathbb{R}\times N$ invariant under the flow so that the restricted developing map is a diffeomorphism onto its image.\\
\\
Form the commutative squares as in Equation \ref{eq16} and Equation \ref{eq18}, and pick a point $y$ in the image of $\overline{\text{dev}}(N)$.  As $\overline{\text{dev}}$ is a local diffeomorphism, the fibre over $y$ is a discrete subset of $N$.  Pick a point $n \in N$ so that $\overline{\text{dev}}(n) = y$, and let $n \in U$ be an open set so that $\overline{\text{dev}}|_{U}$ is a diffeomorphism onto its image $ \overline{\text{dev}}(U)$.  Let $\widetilde{U} = p_{2}^{-1}(U) = \mathbb{R}\times U \subset \mathbb{R} \times N$.\\
\\
Since $\widetilde{U}$ is open in $\mathbb{R}\times N$, to show that $\text{dev}|_{\widetilde{U}}$ is a diffeomorphism onto its image, it suffices to show that the developing map is injective on $\widetilde{U}$.  To this end, let $(t,m)$ and $(s,m')$ be points in $\widetilde{U}$ so that $\text{dev}(t,m) = \text{dev}(s,m')$.  This implies that $(\text{dev}\circ \widetilde{T}_{t})(0,m) = (\text{dev}\circ \widetilde{T}_{s})(0,m')$ and by the definition of $\overline{\text{dev}}$ and Equation \ref{eq17}, one obtains that $\overline{\text{dev}}(m) = \overline{\text{dev}}(m')$.  Since $m,m' \in U$, by construction, this necessitates that $m = m'$.  Consequently $(t,m)$ and $(s,m')$ are in the same $\mathbb{R}$-orbit.  By freeness of the Equation \ref{eq17} and freeness of the $\mathbb{R}$-actions, this necessitates that $s = t$ so $(t,m) = (s,m')$.  Consequently the developing map restricted to $\widetilde{U}$ is diffeomorphism.  
\end{proof}
To summarize, the basic idea of Theorem \ref{thm2} is that there exists a codimension one submanifold $N \subset \widetilde{M}$ transverse to the parallel flow on $\widetilde{M}$ that generates the $\mathbb{R}$-action on $\widetilde{M}$.  In addition, the developing map sends flow lines to flow lines.  The deck transformations, holonomy, and developing map all factor through the corresponding $\mathbb{R}$-actions to yield a local diffeomorphism on $N$ that provide coordinate charts on $N$.  Since $N \subset \widetilde{M}$, these transverse coordinate charts may be saturated by the parallel flow and the saturations are still diffeomorphisms as the flow lines are sent to flow lines, which never loop around and thus provide `large' open subsets so that the restricted developing map is a diffeomorphism onto its image.  Figure \ref{fig:3} provides an illustration of this argument.\\
\\
%
%
%
\begin{figure}
\begin{center}
\includegraphics[scale=0.3]{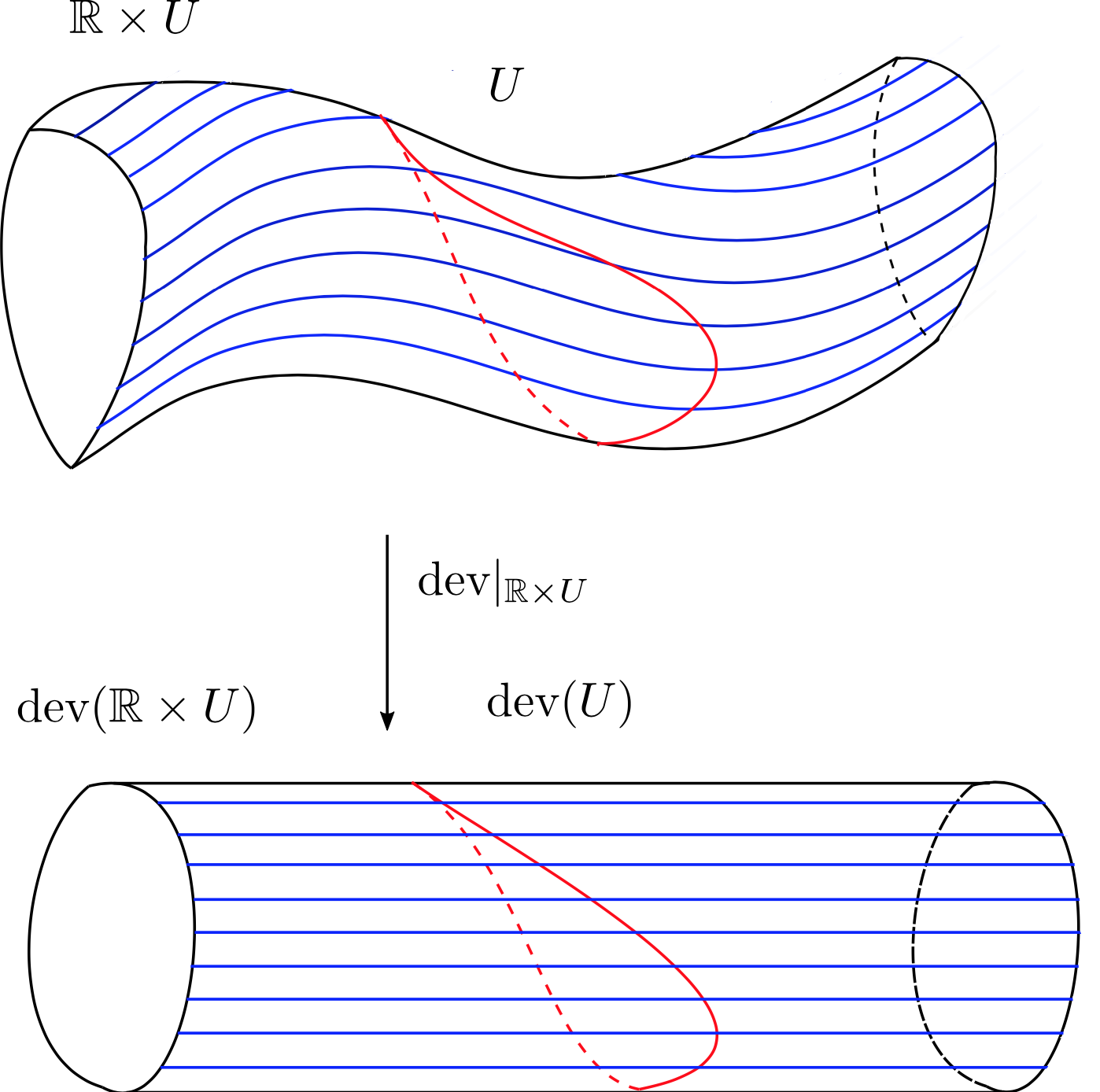}
\end{center}
\caption{Here $U$ is an open subset in $N$ labeled in red for which $\overline{\text{dev}}|_{U}$ is a diffeomorphism onto its image.  The map $\overline{\text{dev}}$ is the induced map on the collapsed blue flow lines.  Saturating $U$ by the $\mathbb{R}$-action yields the open subset $\mathbb{R}\times U \subset \widetilde{M}$ such that the developing map restricted to it is a diffeomorphism.}
\label{fig:3}
\end{figure}
In fact, as one can see in the proof of Theorem \ref{thm2}, the injectivity argument works on any subset $U \subset N$ for which $\overline{\text{dev}}|_{U}$ is a diffeomorphism onto its image.  That said, the failure of the original developing map to be a diffeomorphism onto its image is entirely determined by the failure of $\overline{\text{dev}}$ to be a diffeomorphism.


\section{Radial Flow}
\label{sec:4}

Similar to Section \ref{sec:3}, let $M$ be a closed affine $n$-dimensional manifold with fundamental group $\Gamma = \pi_{1}(M,p)$ acting on the universal cover by deck transformations.  Pick a developing pair $(\text{dev},\text{hol})$ for the affine structure on $M$.  As opposed to the previous section, instead of assuming the linear holonomy fixes a common vector, this section explores some consequences of when the affine holonomy fixes a point in $\mathbb{A}^{n}$.  Pick said fixed point as the origin and make the standard identification of $\mathbb{A}^{n}$ with $\mathbb{R}^{n}$.  Up to conjugation, one may assume the holonomy lies inside the group of linear transformations $\text{GL}(n,\mathbb{R})$.  These class of manifolds are of special interest in the study of geometric structures, so much so that they are provided their own name.
%
%
%
\begin{definition}\label{def3}
A radiant manifold $M$ is an affine manifold whose affine holonomy fixes a point in $\mathbb{A}^{n}$.  This is equivalent to the condition that the affine holonomy is conjugate to a subgroup of $\text{GL}(n,\mathbb{R})$.
\end{definition}
%
%
%
\begin{example}\label{ex6}
The affine structure on the torus given in Example \ref{ex2} provides the torus with a radiant structure, whereas the structure in Example \ref{ex1} is not radiant.  The structure given in Example \ref{ex3}, specifically from Figure \ref{fig:1}, also provides a radiant structure on the torus.  This structure is inequivalent to the one in Example \ref{ex2}, as the holonomy in Example \ref{ex3} is non-cyclic. 
\end{example}
%
%
%
%
%
%
%
\begin{example}\label{ex7}
The structure in Example \ref{ex2} can be generalized in the following fashion.  Consider the puncture euclidean space $\mathbb{R}^{n}\setminus 0$.  Let $H$ be a group generated by a positive homothety induced by some $\lambda > 0$.  Then $\mathbb{R}^{n}/H$ is readily seen to be diffeomorphic to $S^{1}\times S^{n-1}$.  These manifolds, known classically as hopf manifolds, provide a class of examples of radiant manifolds with cyclic holonomy.  An illustration of the identification is provided in Figure \ref{fig:4}.
%
%
%
\begin{figure}
\begin{center}
\includegraphics[scale=0.3]{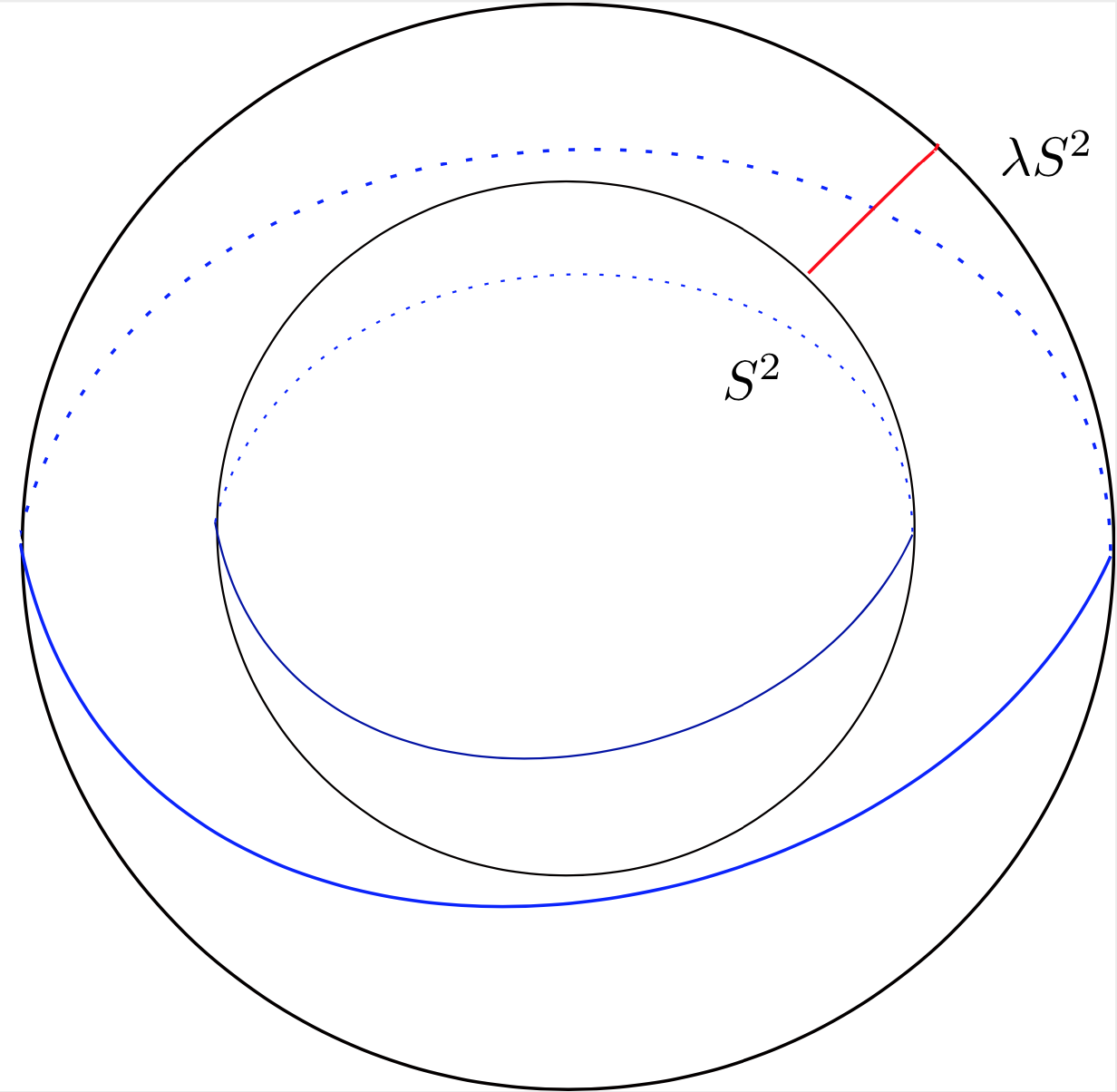}
\end{center}
\caption{An illustration of a three-dimensional hopf-manifold.  The solid space between two concentric spheres is drawn above with an equator in blue.  The action of the homothety identifies the inner sphere with the outer sphere via a dilation.  The curve is red is projected to a circle in the quotient.}
\label{fig:4}
\end{figure}
\end{example}
%
%
%
A standard result of in the theory of geometric structures states that a closed radiant manifold cannot have its fixed point as an element of the developing image \cite{bib8}.  Below is a modified version of the standard argument that is suited for the content of this paper.
%
%
%
\begin{theorem}\label{thm3}
Let $M$ be a closed radiant manifold.  The developing image cannot meet the fixed point of the radiant structure.  
\end{theorem}
%
%
%
\begin{proof}
Fix a developing pair $\text{dev} : \widetilde{M} \longrightarrow \mathbb{A}^{n}$ and $\text{hol} : \Gamma \longrightarrow \text{GL}(n,\mathbb{R})$.  Let $R = -y^{i}\partial/\partial y^{i}$ be the attractive radial vector field on $\mathbb{R}^{n}$.  Routine calculation shows that $R$ is invariant under general linear group, so it may be lifted by the developing map to a $\Gamma$-invariant vector field $\widetilde{R}$ on $\widetilde{M}$.  The vector field $\widetilde{R}$ descends to a vector field on $M$, which is complete by compactness, and thus the corresponding flow on $\widetilde{M}$ is also complete.  Denote the flow on $\widetilde{M}$ by $\widetilde{R}_{t}$ and the corresponding radial flow on $\mathbb{R}^{n}$ by $R_{t}$.  These flows are related by the commutative diagram below
%
%
%
\begin{equation}\label{eq19}
	\begin{tikzcd}
		& \widetilde{M}  \arrow{r}{\widetilde{R}_{t}} \arrow{d}[swap]{\text{dev}}& \widetilde{M} \arrow{d}{\text{dev}}\\
		& \mathbb{R}^{n} \arrow{r}[swap]{R_{t}}& \mathbb{R}^{n}
	\end{tikzcd} \nonumber
\end{equation}
Now, assume that the origin is an element of the developing image.  Then $\text{dev}^{-1}\{0\} \subset \widetilde{M}$ is a discrete subset of stationary points of $\widetilde{R}$.  Choose a collection of pairwise disjoint open sets $U_{i}$ about each element $u_{i} \in\text{dev}^{-1}\{0\}$.  Since $\text{dev}$ is a local diffeomorphism, one may shrink each $U_{i}$ if necessary to assume the developing map restricted to each $U_{i}$ is a diffeomorphism onto an open ball centered at the origin.  For each $t \geq 0$, one has that $R_{t} \text{dev}(U_{i}) \subseteq \text{dev}(U_{i})$ and thus $\widetilde{R}_{t} U_{i} \subseteq U_{i}$.  \\
\\
For each $U_{i}$, let $\widetilde{R}_{\infty}U_{i}$ denote the forward and backward saturation of $U_{i}$ with respect to the radial flow on the universal cover.  Explicitly,
%
%
%
\begin{equation}\label{eq20}
\widetilde{R}_{\infty}U_{i} = \bigcup_{t\in\mathbb{R}} \widetilde{R}_{t}U_{i}
\end{equation}
As $\widetilde{R}_{\infty}U_{i}$ is union of open subsets in $\widetilde{M}$, it is itself an open submanifold of $\widetilde{M}$.  Additionally, the developing map restricted to each $\widetilde{R}_{\infty}U_{i}$ is a diffeomorphism onto its image.  To prove this, it suffices to show that the developing map is injective when restricted to each $\widetilde{R}_{\infty}U_{i}$.  \\
\\
Let $u, v \in \widetilde{R}_{\infty}U_{i}$ so that $\text{dev}(u)= \text{dev}(v)$.  There exists times $t,s \in\mathbb{R}$ and points $u_{i}, v_{i} \in U_{i}$ so that $\text{dev} (\widetilde{R}_{t}u_{i}) = \text{dev} (\widetilde{R}_{s}v_{i})$.  Without loss of generality, let $t-s \geq 0$.  By equivariance of the radial actions, $R_{t-s}\text{dev}(u_{i}) = \text{dev}(v_{i})$.  Since $t-s \geq 0$, one has that $\widetilde{R}_{t-s} U_{i} \subseteq U_{i}$ so $R_{t-s} \text{dev}(U_{i}) \subseteq \text{dev}(U_{i})$.  Because $\text{dev}(\widetilde{R}_{t-s} u_{i}) = \text{dev}(v_{i})$ and $\widetilde{R}_{t-s} U_{i} \subseteq U_{i}$ on which the developing map is a diffeomorphism, one has $\widetilde{R}_{t-s} u_{i} = v_{i}$ so $u = v$ as claimed.  \\
\\
The above paragraphs show that the developing map when restricted to any $\widetilde{R}_{\infty} U_{i}$ is a diffeomorphism onto its image, which is the radial saturation of an open ball about the origin, and consequently a diffeomorphism onto all of $\mathbb{R}^{n}$.  Lemma \ref{lm3} shows that $\widetilde{R}_{\infty} U_{i}$ is closed, and consequently by connectedness of $\widetilde{M}$, is equal to the universal cover.  Thus the developing map is a diffeomorphism onto $\mathbb{R}^{n}$ and defines a complete radial structure.\\
\\
As per consequence there exists a subgroup $H \leq \text{GL}(n,\mathbb{R})$ acting both properly and freely on $\mathbb{R}^{n}$ so that $\mathbb{R}^{n}/H$ is diffeomorphic to $M$.  Since the origin is a fixed point of each element of $\text{GL}(n,\mathbb{R})$ and $H$ acts freely, $H$ must be trivial.  This contradicts the fact that $M$ is compact, and thus the origin is not an element of the developing image.
\end{proof}
As mentioned previously, the above proof yields a corollary that assists the proof of a later theorem.  It is stated here for reference later.  
%
%
%
\begin{corollary}\label{cor1}
Let $N$ be a connected smooth manifold and $F: N \longrightarrow \mathbb{R}^{n}$ be a local diffeomorphism where $0 \in F(N)$.  If the radial action on $\mathbb{R}^{n}$ can be lifted to a complete action on $N$, then $F$ is a diffeormophism onto $\mathbb{R}^{n}$.  
\end{corollary}


\section{Holonomy Acting by Translations on The Invariant Line}
\label{sec:5}
In this section of this paper, a mild generalization of Theorem \ref{thm1} is provided.  The goal of this section is to prove the following theorem.
\begin{theorem}\label{thm4}
Let $M$ be a closed $(n+1)$-dimensional affine manifold whose holonomy leaves invariant an affine line where $n \geq 1$.  If the holonomy acts by pure translations on the invariant line, then the developing image cannot meet the invariant line.  
\end{theorem}
Before beginning the proof it is worth explaining the technical ideas.  As $M$ admits an invariant affline line and the holonomy acts by pure translations on it, one may assume that the holonomy lies in the subgroup of affine automorphisms of the form defined by 
%
%
%
\begin{equation}\label{eq21}
G = 
\left\{
	\left(
		\begin{array}{cc}
		1 & w \\
		0 & A
		\end{array}
	\right)
	\left(
		\begin{array}{c}
		d \\
		0
		\end{array}
	\right) 
	\Bigg|\,
	 d\in\mathbb{R}, w^{T} \in \mathbb{R}^{n}, A \in \text{GL}(n,\mathbb{R})
\right\}
\end{equation}
As per consequence, there is a parallel flow on the universal cover as in Theorem \ref{thm2}.  One can then form the commutative diagram as in Equation \ref{eq18} to obtain the local diffeomorphism $\overline{\text{dev}} : N \longrightarrow \mathbb{R}^{n}$.  Because the induced holonomy in Equation \ref{eq18} acts linearly by the matrices $A$ in Equation \ref{eq21}, there is a complete flow on $N$ that lifts the radial flow on $\mathbb{R}^{n}$.  If the developing image meets the invariant line then $0 \in \overline{\text{dev}}(N)$, and by Corollary \ref{cor1}, $\overline{\text{dev}}$ will define a global diffeomorphism from $N$ onto $\mathbb{R}^{n}$.  Saturating $N$ by the parallel flow yields that the developing map is diffeomorphism onto all of $\mathbb{R}\times\mathbb{R}^{n}$ thus defining a complete structure whose affine holonomy leaves invariant an affine line thus contradicting Theorem 1.  The details of this argument are provided below in the proof.
%
%
%
\begin{proof}
As mentioned in the preceding paragraph, assume the holonomy lies inside the group defined by Equation \ref{eq21}.  By Theorem \ref{thm2}, there exists a complete parallel flow on the universal cover of $M$.  Form the commutative diagram defined by Equation \ref{eq18} with the induced developing map $\overline{\text{dev}} : N \longrightarrow \mathbb{R}^{n}$ and the corresponding induced actions of $\Gamma$ and the holonomy on $N$ and $\mathbb{R}^{n}$ respectively.  \\
\\
Since the holonomy of $M$ lies inside of $G$, the induced holonomy action in Equation \ref{eq18} acts linearly on $\mathbb{R}^{n}$, and thus preserves the attractive radial vector field $R = -y^{i}\partial/\partial y^{i}$ as in Section \ref{sec:4}.  This vector field lifts via $\overline{\text{dev}} : N \longrightarrow \mathbb{R}^{n}$ to a $\Gamma$-invariant vector field $\widetilde{R}$ on $N$.  Identifying the tangent bundle $T(\mathbb{R}\times N)$ with $T\mathbb{R}\oplus TN$ yields a natural lift of $\widetilde{R}$ to a vector field on $\mathbb{R}\times N$, whereby construction, this vector field is also $\Gamma$-invariant.  The $\Gamma$-invariant vector field on $\mathbb{R}\times N$ descends to a vector field on $M$ which is complete by compactness.  The flow on $M$ lifts to a complete flow on $\mathbb{R}\times N$ which leaves each leaf $x\times N$ of $\mathbb{R}\times N$ invariant.  This flow may be thought of as a cylindrical flow which is radial on each leaf.  Above is a figure that illustrates this construction.\\
\\
%
%
%
\begin{figure}
\begin{center}
\includegraphics[scale=0.3]{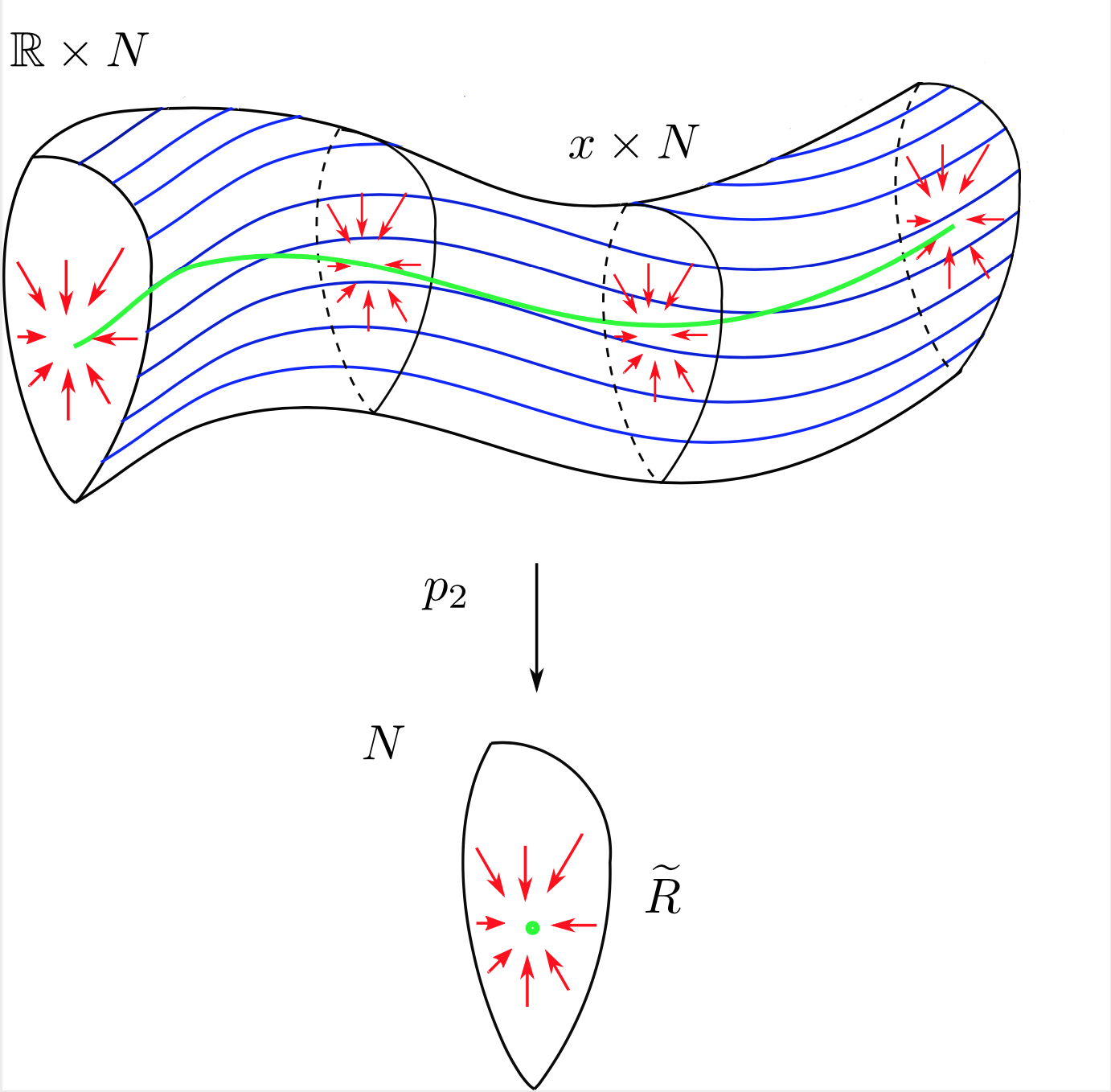}
\end{center}
\caption{An illustration of the cylindrical vector field on $\mathbb{R}\times N$ whose vectors are in red.  Several leaves are drawn transverse to the blue parallel flow lines.  The red inward pointing arrows represent the radial vector field whose flow preserves each leaf $x\times N$.  The radial vector field on $\mathbb{R}\times N$ projects to the vector field $\widetilde{R}$ on $N$ via factor projection $p_{2}$.  This vector field on $N$ is complete, as the cylindrical vector field on $\mathbb{R}\times N$ is complete.  The green parallel flow line in $\mathbb{R}\times N$ is left invariant by the radial flow and represents the parallel flow of a point $n \in N$ so that $\overline{\text{dev}}(n) = 0$.  }
\label{fig:5}
\end{figure}
Since the flow of the lift of $\widetilde{R}$ to $\mathbb{R}\times N$ is complete, by construction the flow of $\widetilde{R}$ is itself complete and thus serves as a lift of the radial flow $R$ on $\mathbb{R}^{n}$ through $\overline{\text{dev}}$.  If the developing image $\text{dev}(\mathbb{R}\times N)$ meets the invariant line, then $0 \in \overline{\text{dev}}(N)$.  Corollary \ref{cor1} implies that $\overline{\text{dev}} : N \longrightarrow \mathbb{R}^{n}$ is a diffeomorphism, and the remark after Theorem \ref{thm2} yields that $\text{dev} : \mathbb{R}\times N \longrightarrow \mathbb{R}\times\mathbb{R}^{n}$ is a diffeomorphism.  Hence $M$ admits a complete affine structure whose holonomy lies inside the group $G$ defined in Equation \ref{eq21}, contradicting Theorem \ref{thm1}.  Thus, the developing image cannot meet the invariant line.   
\end{proof}
As an immediate consequence to the proof of Theorem \ref{thm4} one obtains the following corollary.
%
%
%
\begin{corollary}\label{cor2}
Let $M$ be a closed $(n+1)$-dimensional affine manifold whose holonomy leaves invariant an affine line where $n \geq 1$.  If the holonomy acts by pure translations and reflections on the invariant line, then the developing image cannot meet the invariant line.  
\end{corollary}
\begin{proof}
This is an immediate consequence of the fact that the group defined by Equation \ref{eq21} is an index two subgroup of
%
%
%
\begin{equation}\label{eq22}
\left\{
	\left(
		\begin{array}{cc}
		\pm1 & w \\
		0 & A
		\end{array}
	\right)
	\left(
		\begin{array}{c}
		d \\
		0
		\end{array}
	\right) 
	\Bigg|\,
	 d\in\mathbb{R}, w^{T} \in \mathbb{R}^{n}, A \in \text{GL}(n,\mathbb{R}) \nonumber
\right\}
\end{equation}
\end{proof}
Passing to the double cover $M$ and applying Theorem \ref{thm4} yields the desired result.


\section{Concluding Remarks}
\label{sec:6}
A natural follow up to Theorem \ref{thm4} would be to analyze the situation where the holonomy lies in the extension of the group $G$ as defined in Equation \ref{eq6}.  The author suspects that such manifolds are radiant.  Loosely the idea at hand is the following.  If there is indeed an element of the holonomy of the form
%
%
%
\begin{equation}\label{eq23}
\text{hol}[\gamma] = 	\left(
		\begin{array}{cc}
		r & w \\
		0 & A
		\end{array}
	\right)
	\left(
		\begin{array}{c}
		d \\
		0
		\end{array} \nonumber
	\right) 
\end{equation}
where $r \neq 1$, then without loss of generality, one may conjugate to assume this element of the holonomy acts on the invariant line by scaling, and consequently admits a fixed point, which can be taken as the origin.  It seems likely that the developing image cannot meet this point, much like in the example of a hopf circle.  If that were so, then this would impose certain restrictions about having a non-trivial translational part in the holonomy.  The difficulty in showing this point does not meet the developing image is that $[\gamma] \in \Gamma$ need not stabilize the components of the inverse image of the invariant line under the developing map.  This difficulty is lost in the case where the fundamental group is abelian, but seems like an excessive and unnecessary hypothesis.  \\
\\
The proof of Theorem \ref{thm4} relied largely on the existence of a parallel flow and a `cylindrical' flow on the universal cover of $M$.  These flows can coexist on certain manifolds such as the product of a euclidean circle and a hopf torus as in Example \ref{ex4}.  It is of great interest to the author as to whether or not parallel and radial flows can coexist on compact affine manifolds.  The dynamics of such flows would certainly lead to very interesting examples.


\section{Lemmas}
\label{sec:7}
Here is a collection of some lemmas used throughout the paper.  In this context, all topological spaces are assumed to be smooth manifolds, as is the concern of this paper.  That said, some of these propositions hold in more general contexts such as metric spaces.
\begin{lemma}\label{lm1}
Let $\phi : G \longrightarrow H$ be a  homomorphism of lie groups and let $G$ and $H$ act on the smooth manifolds $X$ and $Y$ respectively.  Assume there exists a diffeomorphism $F : X \longrightarrow Y$ equivariant with respect to $\phi$ in the sense that the below diagram commutes for all $g \in G$.
%
%
%
\begin{equation}\label{eq24}
	\begin{tikzcd}
		&X \arrow{r}{g}  \arrow{d}[swap]{F} &X \arrow{d}{F}\\
		 &Y \arrow{r}[swap]{\phi(g)} & Y
	\end{tikzcd} 
\end{equation}
If $G$ acts properly and freely on $X$ then so too does $\phi(G)$ on $Y$.  In this case, one may form the quotients $X/G$ and $Y/\phi(G)$, which are in turn diffeomorphic.  
\end{lemma}
\begin{proof}
Begin by assuming that $G$ acts freely on $X$.  Let $\phi(g)y = y$ for some $g \in G$ and $y \in Y$.  Since $F$ is a diffeomorphism, there's an $x \in X$ so that $F(x) = y$ and so $\phi(g)F(x) = F(x)$ which by equivariance is equivalent to $F(gx) = F(x)$.  Since $G$ is a diffeomorphism, this necessitates that $gx = x$, and thus by freeness of $G$ on $X$, $g = 1$, so $\phi(g) = 1$.  A similar argument shows the homomorphism $\phi$ is injective, so $G$ and $\phi(G)$ are diffeomorphic as manifolds. \\
\\
%
%
%
Let $G$ act properly on $X$.  Pick sequences $\phi(g_{i}) \in \phi(G)$ and $y_{i} \in Y$ so that $\phi(g_{i})y_{i}$ converges to a point $q \in Y$ and $y_{i}$ converges to $y \in Y$.  To show properness it suffices to show a subsequence of $\phi(g_{i})$ converges, as is stated in John M. Lee's Smooth Manifolds \cite{bib7}.  Since $F$ is a diffeomorphism, there's a unique sequence of $x_{i} \in X$ converging to an $x \in X$ so that $F(x_{i}) = y_{i}$ and $F(x) = y$.  Equivariance yields that $\phi(g_{i})y_{i} = \phi(g_{i})F(x_{i}) = F(g_{i}x_{i})$ which converges to $q \in Y$, and thus $g_{i}x_{i}$ converges to some $p \in X$.  Since $g_{i}x_{i}$ converges to $p \in X$ and $x_{i}$ converges to $x \in X$, properness of $G$ on $X$ yields a convergent subsequence of $g_{i}$ to $g \in G$.  Continuity of the lie group homomorphism $\phi : G \longrightarrow H$ provides a convergent subsequence $\phi(g_{i})$ converging to $\phi(g) \in \phi(G)$, and thus the action of $\phi(G)$ on $Y$ is proper.  \\
\\
As both actions of $G$ on $X$ and $\phi(G)$ on $Y$ are free and proper, one may form their smooth quotient manifolds $X/G$ and $Y/\phi(G)$.  Denote their corresponding projections by $p : X \longrightarrow X/G$ and $q : Y \longrightarrow Y/\phi(G)$.  One may then form the commutative square below
%
%
%
\begin{equation}\label{eq25}
	\begin{tikzcd}
		&X \arrow{r}{F} \arrow{d}[swap]{p}&Y\arrow{d}{q}\\
		&X/G \arrow{r}[swap]{\overline{F}} & Y/\phi(G)
	\end{tikzcd}
\end{equation}
The map $\overline{F}$ is well defined as Equation \ref{eq24} ensures $F$ maps orbits to orbits.  As $q\circ F$ is surjective, so too is $\overline{F}$.  It is injective because if $\overline{F}(Gu) = \overline{F}(Gv)$ for some $Gu, Gv \in X/G$, then there exists $u,v \in X$ and a $g \in G$ so that $\phi(g)F(u) = F(v)$.  Equivariance implies $F(gu) = F(v)$ and since $F$ is a diffeomorphism, $u$ and $v$ are in the same orbit, thus $Gu = Gv$, so $\overline{F}$ is a bijection.  \\
\\
Since $X$ and $Y$ are the same dimension, as are $G$ and $\phi(G)$, the map $\overline{F}$ is a smooth bijective local diffeomorphism, and thus a diffeomorphism.
\end{proof}
Lemma \ref{lm1} provides the following result frequently used in the study of geometric structures.  
%
%
%
\begin{corollary}\label{cor3}
Let $M$ be a complete affine $n$-dimensional manifold with fundamental group $\Gamma = \pi_{1}(M,p)$.  Fix a developing pair $\text{dev} :\widetilde{M} \longrightarrow \mathbb{A}^{n}$ and $\text{hol} : \Gamma \longrightarrow \text{Aff}(n,\mathbb{R})$.  Then $M$ is diffeomorphic to $\mathbb{A}^{n}/H$ where $H$ is the image of the holonomy homomorphism.  
\end{corollary}
%
%
\begin{proof}
Since the developing map is a diffeomorphism onto $\mathbb{A}^{n}$ and $M$ is diffeomorphic $\widetilde{M}/\Gamma$, where $\Gamma$ is the group of deck transformations that acts both properly and freely on the universal cover, Lemma \ref{lm1}, applied the developing map and holonomy homomorphism yield that $\widetilde{M}/\Gamma$ and therefore $M$, are both diffeomorphic to $\mathbb{A}^{n}/H$.  
\end{proof}
The following statement has a proof similar to that of Lemma \ref{lm1}, and is used in the construction of the quotient manifolds in Section \ref{sec:3}.  Its proof is omitted as it is nearly identical to that of the previous lemma.  
%
%
\begin{lemma}\label{lm2}
Let $G$ be a lie group acting on smooth manifolds $X$ and $Y$ and let $F : X \longrightarrow Y$ be a smooth map equivariant with respect to the $G$-actions.  If the action of $G$ on $Y$ is free and proper, then so too is the action of $G$ on $X$.  In this case one may form the quotients $X/G$ and $Y/G$ for which $F$ descends to a smooth map $\overline{F} : X/G \longrightarrow Y/G$.  
\end{lemma}
%
\begin{lemma}\label{lm3}
Let $N$ and $P$ be smooth manifolds and $F : N \longrightarrow P$ be an open map where $N$ is connected.  If there exists an open submanifold $U \subseteq N$ for which $F|_{U} : U \longrightarrow P$ is a diffeomorphism, then $N = U$.  
\end{lemma}
This following lemma finds it use in the proof of Theorem \ref{thm3}, to show the open submanifold $\widetilde{R}_{\infty}U_{i}$ defined by Equation \ref{eq20} is equal to all of $\widetilde{M}$.  In this case the developing map fulfills the role of the open map as stated in the lemma.  
\begin{proof}
It suffices to show that $U$ is closed.  Let $u_{k}$ be a sequence of points in $U$ converging to some point $n \in N$.  By continuity of $F$, the sequence $F(u_{k})$ converges to $F(n) \in P$.  Since $F|_{U} : U \longrightarrow P$ is a diffeomorphism, there's a unique $u \in U$ so that $F(n) = F(u)$.  The claim is that $n = u$.  \\
\\
Let $u \in V$ be an open neighborhood in $N$ about $u$.  As $U$ is open in $N$, one may shrink $V$ sufficiently small so that $u \in V\subseteq U$.  Because the sequence $F(u_{k})$ converges to $F(n) = F(u)$ and $F(V)$ is an open subset of $P$ about $F(u)$, there exists a sufficiently large $K \in \mathbb{N}$ so that $F(u_{k}) \in F(V)$ for all $k \geq K$.  Since $F|_{U} : U \longrightarrow P$ is a diffeomorphism, each $u_{k} \in U$, and $V \subset U$, it follows that $u_{k} \in V$ for $k \geq K$.  Thus $u_{k}$ converges to both $u$ and $n$, and by uniqueness of limits, $u = n$.  Consequently, $U$ is a non-empty closed and open subset of $N$, and by connectedness $U = N$.  
\end{proof}


%
%

\begin{acknowledgements}
I would like to thank Dr. William Goldman for his input and assistance throughout the formation of this paper.  As always, I thank my wonderful friends whose creativity and good deeds inspire my day to day life.  Finally, a very special thank you to Yon Hui and Chaz Daly, and Rose of Sharon and Anatoly Bourov-Daly for a list of kindnesses several times the length of this paper.
\end{acknowledgements}

%
%

\bibliographystyle{spmpsci}      
\bibliography{mybibliography}{}   


\end{document}